\documentclass[12pt, twoside]{article}
\usepackage{amsmath,amsthm,amssymb}
\usepackage{times}
\usepackage{enumerate}
\usepackage{bm}
\usepackage[all]{xy}
\usepackage{mathrsfs}
\usepackage{amscd}
\usepackage{color}
\def\titlerunning#1{\gdef\titrun{#1}}
\makeatletter
\def\author#1{\gdef\autrun{\def\and{\unskip, }#1}\gdef\@author{#1}}
\def\address#1{{\def\and{\\\hspace*{18pt}}\renewcommand{\thefootnote}{}%
		\footnote {#1}}%
	\markboth{\autrun}{\titrun}}
\makeatother
\def\email#1{e-mail: #1}

\def\keywords#1{\par\medskip
	\noindent\textbf{Keywords.} #1}

\newtheorem{theorem}{Theorem}[section]
\newtheorem{corollary}[theorem]{Corollary}
\newtheorem{lemma}[theorem]{Lemma}
\newtheorem{proposition}[theorem]{Proposition}
\theoremstyle{definition}

\newtheorem{remark}[theorem]{Remark}

\numberwithin{equation}{section}

\newcommand{\la}{\lambda}
\newcommand{\La}{\Lambda}
\newcommand{\diam}{\operatorname{diam}}
\newcommand{\ANCO}{\text{ANCO}}

\xyoption{all}

\def \la {\lambda}
\def \La {\Lambda}

\begin{document}
	\baselineskip=17pt
\titlerunning{Euler Characteristic of Closed Manifolds with ANCO}
\title{Euler Characteristic of Closed Manifolds with Almost Nonnegative Curvature Operator}

\author{Jing-Bin Cai}

\date{}

\maketitle

\address{J.~Cai: School of Mathematical Sciences, University of Science and Technology of China, Hefei, Anhui 230026, People's Republic of China; \email{binge@mail.ustc.edu.cn}}

\begin{abstract}
	We study closed manifolds with almost nonnegative curvature operator and address a question of Herrmann--Sebastian--Tuschmann concerning the sign of their Euler characteristic. Our main result shows that if a closed $2n$-dimensional manifold admits an almost nonnegative curvature operator together with a uniform upper bound on the curvature operator, then its Euler characteristic is nonnegative. In addition, under an ANCO-type condition and assuming that the fundamental group is infinite, we prove vanishing results for the Euler characteristic, the signature, and, in the spin case, the $\widehat{A}$-genus, extending recent work of Chen--Ge--Han from almost nonnegative Ricci curvature to the curvature-operator setting.
\end{abstract}


\section{Introduction}

A central theme in Riemannian geometry is to understand how curvature conditions constrain the topology of manifolds. In this context, the curvature operator plays a particularly rigid role. Let $(X,g)$ be a closed $n$-dimensional Riemannian manifold, and let $\la_{1} \leq \cdots \leq \la_{\binom{n}{2}}$ denote the eigenvalues of its curvature operator. Classical work of Berger \cite{Ber61} and Meyer \cite{Mey71} established vanishing results for the Betti numbers of manifolds with positive curvature operator (i.e.\ $\la_{1}>0$); in particular, Meyer showed that such manifolds are rational (co)homology spheres. More recently, Petersen and Wink \cite{PW21} proved a more general vanishing theorem and an estimate for the $p$-th Betti number of closed $n$-dimensional Riemannian manifolds satisfying
\[
\la_{1} + \cdots + \la_{n-p} \geq 0.
\]
As a consequence, they showed that manifolds whose curvature operator is $\lceil\frac{n}{2} \rceil$-positive are rational homology spheres.

Motivated by the classification of manifolds with nonnegative curvature operator and by collapsing phenomena, Herrmann, Sebastian, and Tuschmann introduced the notion of almost nonnegative curvature operator. A closed smooth manifold $X$ is said to admit an \emph{almost nonnegative curvature operator} (\ANCO) if it admits a sequence of Riemannian metrics $\{g_i\}$ such that all eigenvalues $\la_j(g_i)$ of the associated curvature operator and the diameter satisfy
\[
\la_j(g_i) \geq -\frac{1}{i} \quad \text{for all } j, \qquad \diam(X, g_i) \leq 1.
\]
We will refer to such manifolds as ANCO manifolds. This notion is a natural ``almost nonnegative'' analogue of nonnegative curvature operator and was introduced in \cite{HST13,Tus16}. While the recent work of B\"ohm and Wilking \cite{BW08} completed the classification of manifolds with nonnegative and positive curvature operator---a program initiated by Berger more than half a century ago---the case of ANCO manifolds remains largely unexplored.

From the viewpoint of collapsing theory, ANCO manifolds arise naturally as well. Lott \cite{Lot14} developed an initial structure theory for manifolds collapsing under a lower bound on the curvature operator, extending earlier results of Cheeger and Gromov on $F$-structures under two-sided sectional curvature bounds to this setting. His work suggests that, in a more general collapsing picture that incorporates the various directions present in the Cheeger--Fukaya--Gromov theory of $N$-structures, ANCO manifolds should play a role analogous to that of almost flat manifolds \cite{Gro78,Ruh82}.

It is well known that the Euler characteristic of a manifold with nonnegative curvature operator is nonnegative (see Proposition~\ref{P2}). In \cite{HST13}, Herrmann, Sebastian, and Tuschmann formulated the following natural question, which asks to what extent this nonnegativity property persists for ANCO manifolds:

\begin{quote}
	\textbf{Question 4.6.} Do manifolds with almost nonnegative curvature operator have nonnegative Euler characteristic?
\end{quote}

For general Riemannian manifolds, Huang and Tan \cite{HT23} obtained partial affirmative answers. In particular, they proved the following result, which we quote in a slightly simplified form.

\begin{theorem}[\cite{HT23}]\label{thm:HT}
	Let $(X,g)$ be a closed smooth $2n$-dimensional Riemannian manifold with $b_1(X) > 0$. There exists a uniformly positive constant $c = c(n) > 0$ such that if the eigenvalues of the curvature operator satisfy
	\begin{equation}\label{eq:main}
		\big( \la_1(g) + \cdots + \la_n(g) \big) \cdot \diam^2(g) \geq -c(n),
	\end{equation}
	then $\chi(X) = 0$.
\end{theorem}	

As a first consequence, one obtains a complete affirmative answer to Question~4.6 in dimension $4$. Indeed, if $\dim X = 4$, then
\[
\chi(X) = 2 + b_2(X) - 2b_1(X).
\]
If $b_1(X) = 0$, then $\chi(X) \geq 2$; if $b_1(X) \geq 1$, then Theorem~\ref{thm:HT} implies $\chi(X) = 0$. Thus we have:

\begin{corollary}\label{cor:dim4}
	Let $(X,g)$ be a closed smooth $4$-dimensional Riemannian manifold. There exists a uniformly positive constant $c > 0$ such that if the eigenvalues of the curvature operator of $g$ satisfy
	\begin{equation}\label{eq:dim4}
		\big( \la_1(g) + \la_2(g) \big) \cdot \diam^2(g) \geq -c,
	\end{equation}
	then $\chi(X) \geq 0$.
\end{corollary}

The main purpose of this note is to provide further evidence for Question~4.6 under a natural additional hypothesis, namely the existence of a uniform upper bound on the curvature operator. Following an idea of Lott~\cite{Lot14}, we show that in even dimensions such a two-sided control on the curvature operator forces the Euler characteristic to be nonnegative.

\begin{theorem}\label{thm:main}
	Let $X$ be a closed smooth $2n$-dimensional Riemannian manifold. For any $\La > 0$, there exists $\varepsilon(n,\La) > 0$ such that if the curvature operator of $g$ satisfies
	\begin{equation}\label{eq:bound}
		-\varepsilon(n,\La) \leq \la_1(g) \cdot \diam^2(g) \leq \cdots \leq \la_{n(2n-1)}(g) \cdot \diam^2(g) \leq \La,
	\end{equation}
	then $\chi(X) \geq 0$.
\end{theorem}

In particular, Theorem~\ref{thm:main} gives a positive answer to Question~4.6 for ANCO manifolds whose curvature operators admit a uniform upper bound in the above sense. The proof proceeds by a compactness argument: under the two-sided bounds~\eqref{eq:bound}, a Cheeger--Gromov type convergence theorem yields a limit metric with nonnegative curvature operator, and the nonnegativity of the Euler characteristic then follows from the Chern--Gauss--Bonnet formula (see Proposition~\ref{P2}).

In addition, we establish a vanishing theorem for genera under a natural ANCO-type hypothesis and an assumption on the fundamental group. For a closed smooth $2n$-dimensional manifold $X$ with infinite fundamental group, we show that suitable almost nonnegativity conditions on the curvature operator imply the vanishing of the Euler characteristic, the signature, and, in the spin case, the $\widehat{A}$-genus (see Theorem~\ref{thm:vanishing-genera}). This result can be viewed as a curvature-operator analogue of the vanishing theorems obtained by Chen, Ge, and Han under almost nonnegative Ricci curvature~\cite{CGH24}.

The paper is organized as follows. In Section~2 we recall the relevant Bochner--Weitzenb\"ock formulas and establish a vanishing theorem for genera under an almost nonnegative curvature operator condition. In Section~3 we combine compactness results for Riemannian manifolds with bounded curvature operator and Weyl's perturbation theorem to prove Theorem~\ref{thm:main}.

\section{Curvature operator and Weitzenb\"ock formulas}

Let $(X,g)$ be a closed $n$-dimensional Riemannian manifold, and let its curvature tensor be given by
\[
R(X,Y)Z = \nabla_Y\nabla_X Z - \nabla_X\nabla_Y Z + \nabla_{[X,Y]}Z.
\]
Given a tensor field $T$, the Weitzenb\"ock curvature term acting on $T$ is defined by
\[
Ric(T)(X_1,\dots,X_k) = \sum_{i=1}^k \sum_{j=1}^n \left(R(X_i, e_j)T\right)(X_1, \dots, e_j, \dots, X_k),
\]
where $\{e_j\}_{j=1}^n$ is a local orthonormal frame. This operator generalizes the Ricci tensor, in the sense that it coincides with the usual Ricci endomorphism when applied to vector fields and $1$-forms. The Hodge Laplacian $\Delta_d$ satisfies the Weitzenb\"ock formula
\[
\Delta_d T = \nabla^*\nabla T + Ric(T).
\]
The curvature operator of the metric $g$ is the bundle endomorphism
\[
\mathfrak{R} : \Lambda^2 TX \to \Lambda^2 TX, \qquad (\mathfrak{R}(\omega))_{ij} = \sum_{k,l} R_{ijkl} \omega_{kl},
\]
which is self-adjoint with respect to the natural inner product on $\Lambda^2 TX$.

We recall the following curvature estimate, which is the basic linear-algebraic ingredient in the Bochner technique of Petersen and Wink.

\begin{proposition}[\cite{PW21}, Lemma~2.1]\label{P1}
	Let $(X,g)$ be a closed $m$-dimensional Riemannian manifold, and let
	\[
	\lambda_1 \leq \cdots \leq \lambda_{\binom{m}{2}}
	\]
	be the eigenvalues of its curvature operator. Fix an integer $1 \leq p \leq \lfloor \tfrac{m}{2}\rfloor$, and suppose that for some $\kappa \le 0$ one has
	\[
	\frac{\lambda_1 + \cdots + \lambda_{m-p}}{m-p} \geq \kappa.
	\]
	Then for any $\alpha \in \Omega^k(X)$ with $k \leq p$ or $k \geq m-p$, one has
	\[
	g(Ric(\alpha), \alpha) \geq \kappa\, k(m-k)\,|\alpha|^2.
	\]
\end{proposition}

\begin{remark}
	In \cite[Lemma~2.1]{PW21} the estimate is proved under the same assumption
	$\kappa\le 0$, which is precisely the range relevant for almost nonnegativity
	conditions on the curvature operator. A closely related formulation, together
	with a complete proof in the context of curvature operators, can also be found
	in \cite[Corollary~2.3]{HT23}. In the present paper we will only apply
	Proposition~\ref{P1} with parameters $\kappa\le 0$.
\end{remark}

The first main theorem in \cite{PW21} introduces a family of nested curvature conditions that lead to various vanishing results for Betti numbers. Recall that the curvature operator is called \emph{$(n-l)$-nonnegative} (resp.\ \emph{$(n-l)$-positive}) if the sum of its first $n-l$ eigenvalues is nonnegative (resp.\ positive), i.e.,
\[
\lambda_1 + \cdots + \lambda_{n-l} \geq 0 \quad (\text{resp. } >0).
\]

\begin{theorem}[\cite{PW21}, Theorems~A and B]\label{thm:PW}
	Let $n \geq 3$ and $1 \leq l \leq \lfloor\frac{n}{2} \rfloor$. If $(X,g)$ is a closed $n$-dimensional Riemannian manifold with $(n-l)$-nonnegative curvature operator, then every harmonic $k$-form is parallel. Moreover, if the curvature operator is $(n-l)$-positive, then
	\[
	b_1(X) = \cdots = b_l(X) = 0 \quad \text{and} \quad b_{n-l}(X) = \cdots = b_{n-1}(X) = 0.
	\]
\end{theorem}

Following Petersen and Wink, we say that a closed $n$-dimensional smooth manifold $X$ admits an \emph{almost $(n-l)$-nonnegative curvature operator} (or \emph{$(n-l)$-ANCO} for short) if there exists a sequence of Riemannian metrics $\{g_i\}$ on $X$ such that
\[
\lambda_1(g_i) + \cdots + \lambda_{n-l}(g_i) \geq -\frac{n-l}{i}, \qquad \diam(X,g_i) \leq 1.
\]
If a sequence $\{g_i\}$ on a closed smooth manifold $X$ satisfies
\[
Ric(g_i) \geq -\frac{1}{i}g_i, \qquad \diam(X,g_i) \leq 1 \quad \text{for all } i,
\]
then we say that $(X, \{g_i\})$ has \emph{almost nonnegative Ricci curvature}. It is clear that an $(n-l)$-ANCO manifold has almost nonnegative Ricci curvature (cf.\ \cite{Ber61,Gro81,KPT10,KW11}).

Under almost nonnegative Ricci curvature and an additional hypothesis on the fundamental group, Chen, Ge, and Han obtained several vanishing theorems for genera \cite{CGH24}. Inspired by their work, we obtain the following result, which adapts their arguments to the curvature-operator setting by means of Proposition~\ref{P1}.

\begin{theorem}\label{thm:vanishing-genera}
	Let $X$ be a closed smooth $2n$-dimensional manifold with infinite fundamental group. Suppose there exists a sequence of Riemannian metrics $\{g_i\}$ on $X$ such that
	\begin{equation}\label{eq:sequence}
		\left( \lambda_1(g_i) + \cdots + \lambda_n(g_i) \right) \cdot \diam^2(X,g_i) \geq -\frac{n}{i},
	\end{equation}
	for all $i$. Then $\chi(X) = 0$ and $\sigma(X) = 0$. Moreover, if $X$ is spin, then $\widehat{A}(X) = 0$.
\end{theorem}

\begin{proof}
	The argument follows that of \cite[Theorem~1.1]{CGH24}, with Proposition~\ref{P1} providing the required curvature estimates.
	
	By rescaling, we may assume that
	\[
	\diam(X,g_i) = 1 \quad \text{for all } i.
	\]
	Then the assumption~\eqref{eq:sequence} becomes
	\[
	\lambda_1(g_i) + \cdots + \lambda_n(g_i) \geq -\frac{n}{i}.
	\]
	Since $\dim X = 2n$, we can apply Proposition~\ref{P1} with $m = 2n$ and $p = n$. For each $i$ we define
	\[
	\kappa_i := \frac{\lambda_1(g_i) + \cdots + \lambda_n(g_i)}{n} \ge -\frac{1}{i},
	\]
	so that in particular $\kappa_i \le 0$ for all $i$. Hence Proposition~\ref{P1} yields
	\[
	g(Ric(g_i)(\alpha), \alpha) \;\ge\; \kappa_i\, k(2n-k)\,|\alpha|^2
	\]
	for any $\alpha \in \Omega^k(X)$ with $k \leq n$ or $k \geq n$. In particular, for such $k$ we have the uniform estimate
	\[
	g(Ric(g_i)(\alpha), \alpha) \;\ge\; -\frac{C(n)}{i}\,|\alpha|^2,
	\]
	where $C(n)>0$ depends only on the dimension (for instance, one may take
	$C(n) = \max_{1\le k\le 2n} k(2n-k)$).
	
	Consider the operator
	\[
	d + d^* : \Omega^{\text{even}}(X) \to \Omega^{\text{odd}}(X),
	\]
	whose index equals the Euler characteristic of $X$, and, in dimension $4n$, the operator
	\[
	d + d^* : \Lambda^+ TX \to \Lambda^- TX,
	\]
	whose index equals the signature of $X$. The Weitzenb\"ock formulas for the corresponding Hodge Laplacians express
	\[
	(d + d^*)^2 = \nabla^* \nabla + \mathcal{R},
	\]
	where the curvature term $\mathcal{R}$ is a linear combination of $Ric$ and the full curvature operator acting on differential forms. The lower bound above on $g(Ric(g_i)(\alpha),\alpha)$ for $k\le n$ and $k\ge n$ implies that the curvature term $\mathcal{R}$ satisfies
	\[
	\langle \mathcal{R}_{g_i}\beta, \beta\rangle \;\ge\; -\frac{C(n)}{i}\,|\beta|^2
	\]
	for all $\beta$ in the relevant form bundles, with the same constant $C(n)$. This is exactly the curvature inequality needed in the index-theoretic argument of \cite[Theorem~1.1]{CGH24}. Using the almost nonnegativity of the curvature term and the fact that $\pi_1(X)$ is infinite, one shows that the $L^2$-indices of the lifted operators on the universal cover vanish, which forces
	\[
	\chi(X) = 0 \quad \text{and} \quad \sigma(X) = 0.
	\]
	We refer to \cite[Section~2]{CGH24} for the detailed analytic implementation of this argument.
	
	If $X$ is spin, the conclusion for the $\widehat{A}$-genus follows from \cite[Theorem~1.2]{CGH24}. Indeed, applying Proposition~\ref{P1} with $m=2n$, $p=n$ and $k=1$, we obtain
	\[
	g(Ric(g_i)(\alpha),\alpha) \;\ge\; \kappa_i\,(2n-1)\,|\alpha|^2 \;\ge\; -\frac{C'(n)}{i}\,|\alpha|^2
	\]
	for all $1$-forms $\alpha$, where $C'(n) > 0$ depends only on $n$. Equivalently,
	\[
	Ric(g_i) \;\ge\; -\frac{C'(n)}{i}\, g_i.
	\]
	Thus the sequence $\{g_i\}$ has almost nonnegative Ricci curvature in the sense of \cite{CGH24}, and \cite[Theorem~1.2]{CGH24} implies that $\widehat{A}(X) = 0$.
\end{proof}

\section{Eigenvalues of curvature operators with a uniform upper bound}
	
	In this section we combine basic matrix estimates with a Cheeger--Gromov type compactness theorem to prove Theorem~\ref{thm:main}. We begin with some linear algebraic preliminaries.
	
	Let $A$ be an $n \times n$ Hermitian matrix. Its operator norm is defined by
	\[
	\|A\| = \sup_{|x|=1} |\langle x, Ax \rangle|.
	\]
	Another useful norm is the Frobenius norm:
	\[
	\|A\|_2 = \left( \sum_{i,j=1}^n |a_{ij}|^2 \right)^{1/2} = \left( \operatorname{tr}(A^*A) \right)^{1/2}.
	\]
	These two norms satisfy
	\[
	\|A\| \le \|A\|_2 \le \sqrt{n} \,\|A\|.
	\]
	
	We also recall Weyl's perturbation theorem for Hermitian matrices.
	
	\begin{theorem}[\cite{Bha97}, Theorem~VI.2.1]\label{T3}
		Let $A$ and $B$ be Hermitian matrices with eigenvalues $\lambda_1(A) \ge \cdots \ge \lambda_n(A)$ and $\lambda_1(B) \ge \cdots \ge \lambda_n(B)$, respectively. Then
		\[
		\max_j |\lambda_j(A) - \lambda_j(B)| \le \|A - B\|.
		\]
	\end{theorem}
	
	Next we recall a Cheeger--Gromov type convergence theorem due to Kasue.
	
	\begin{theorem}[\cite{Kas89}]\label{thm:Cheeger-Gromov}
		Let $(X_i, g_i)$ be a sequence of closed Riemannian manifolds such that
		\[
		|K(g_i)| \le \Lambda, \quad \operatorname{Vol}(g_i) \ge v > 0, \quad \operatorname{diam}(X_i,g_i) \le D,
		\]
		where the constants $\Lambda, v, D$ are independent of $i$. Then, after passing to a subsequence, $(X_i, g_i)$ converges to a metric space $X$ such that:
		\begin{enumerate}
			\item[(i)] $X$ is a differentiable manifold;
			\item[(ii)] for all sufficiently large $i$, there exists a diffeomorphism $f_i : X \to X_i$;
			\item[(iii)] the pullback metrics $f_i^* g_i$ converge in the $C^{1,\alpha}$-topology (respectively, $L^p_2$-topology) to a $C^{1,\alpha}$ (resp.\ $L^p_2$) Riemannian metric $g_\infty$ on $X$, for any $\alpha \in (0,1)$ (resp.\ $p > 1$).
		\end{enumerate}
	\end{theorem}
	
	For later use we record the following consequence of Theorem~\ref{thm:Cheeger-Gromov} and standard elliptic regularity, formulated in our present notation.
	
	\begin{lemma}\label{lem:curvature-convergence}
		Let $(X_i,g_i)$ be a sequence of closed Riemannian $n$-manifolds satisfying
		\[
		|K(g_i)| \le \Lambda, \qquad \operatorname{Vol}(X_i,g_i) \ge v_0 > 0,
		\qquad \operatorname{diam}(X_i,g_i) \le D,
		\]
		with constants $\Lambda,v_0,D$ independent of $i$. Suppose that, after passing to a subsequence and pulling back by diffeomorphisms, $g_i$ converges in the $C^{1,\alpha}$-topology to a Riemannian metric $g_\infty$ on a fixed manifold $X$, for all $0<\alpha<1$. Then the following hold:
		\begin{enumerate}
			\item[(a)] For every $p\in[1,\infty)$, we may (after passing to a further subsequence if necessary) assume that
			\[
			g_i \to g_\infty \quad\text{in the Sobolev space }L^p_2(X),
			\]
			that is, $g_i$ and its weak derivatives up to order two converge to those of $g_\infty$ in $L^p$.
			\item[(b)] Let $K_i$ and $K_\infty$ denote the curvature tensors of $g_i$ and $g_\infty$, respectively. Then $K_\infty\in L^p(X)$ for all $p>1$, and
			\[
			K_i \to K_\infty \quad\text{in }L^p(X)
			\]
			for every $p\ge1$.
		\end{enumerate}
	\end{lemma}
	
	\begin{remark}
		The proof uses harmonic coordinates and elliptic estimates for the metric
		components, together with the uniform curvature and volume bounds; the
		convergence of the curvature tensors then follows from the fact that the
		curvature is an algebraic expression in the second derivatives of the metric.
		We refer for instance to Anderson~\cite[Section~2]{And90} and to the discussion
		in Lott~\cite[Section~2]{Lot00} for the detailed analytic background and omit
		the proof here.
	\end{remark}

	We now show that a uniform bound on the eigenvalues of the curvature operator yields a uniform lower bound on the volume, provided the Euler characteristic is nonzero.
	
	\begin{lemma}\label{L1}
		Let $X$ be a closed smooth $2n$-dimensional Riemannian manifold with $\chi(X)\neq 0$. Suppose there exists a constant $\Lambda > 0$ such that the eigenvalues of the curvature operator satisfy
		\[
		-\Lambda \le \lambda_1 \le \cdots \le \lambda_{n(2n-1)} \le \Lambda.
		\]
		Then there exists $v = v(n,\Lambda) > 0$ such that $\operatorname{Vol}(X,g) \ge v$.
	\end{lemma}
	
	\begin{proof}
		Recall that the curvature operator $\mathfrak{R}$ is symmetric, since
		\[
		g(\mathfrak{R}(e_i \wedge e_j), e_k \wedge e_l) = R(e_i, e_j, e_k, e_l) = R(e_k, e_l, e_i, e_j)
		\]
		for any local orthonormal frame $\{e_i\}$. Hence all eigenvalues of $\mathfrak{R}$ are real, and
		\[
		\|\mathfrak{R}\|_2 \le \sqrt{n(2n-1)} \,\|\mathfrak{R}\| \le \sqrt{n(2n-1)} \,\Lambda.
		\]
		In particular, the sectional curvature is uniformly bounded:
		\[
		|K| \le C(n)\,\Lambda
		\]
		for some constant $C(n)>0$ depending only on the dimension.
		
		The Euler characteristic can be expressed in terms of the curvature tensor as
		\[
		\chi(X) = \int_X P(K(X,g))\, d\mathrm{Vol}_g,
		\]
		where $P$ is an explicit homogeneous polynomial in the components of the curvature tensor (see \cite{LM89}). Since $\chi(X)\neq 0$ and $P(K(X,g))$ is uniformly bounded in absolute value by a constant depending only on $n$ and $\Lambda$, it follows that there exists $v(n,\Lambda) > 0$ such that $\operatorname{Vol}(X,g) \ge v(n,\Lambda)$.
	\end{proof}
	
	As an immediate consequence we obtain the following compactness statement.
	
	\begin{corollary}\label{C2}
		Let $\widetilde{\mathcal{M}}(\Lambda, D)$ denote the set of closed Riemannian $2n$-manifolds $(X,g)$ satisfying
		\[
		|\lambda_k(g)| \le \Lambda, \quad \operatorname{diam}(X,g) \le D, \quad \chi(X) \ne 0.
		\]
		Then $\widetilde{\mathcal{M}}(\Lambda, D)$ is precompact in the $C^{1,\alpha}$- and $L^p_2$-topologies (up to diffeomorphism).
	\end{corollary}
	
	\begin{proof}
		Given $(X,g)\in\widetilde{\mathcal{M}}(\Lambda, D)$, Lemma~\ref{L1} provides a uniform lower bound $\operatorname{Vol}(X,g)\ge v(n,\Lambda)>0$. Together with the diameter bound and the uniform bound on the eigenvalues of the curvature operator, we obtain a uniform sectional curvature bound. The conclusion then follows from Theorem~\ref{thm:Cheeger-Gromov}.
	\end{proof}
	
	\begin{proposition}[\cite{BK78, Kul72}]\label{P2}
		Let $X$ be a closed smooth $2n$-dimensional Riemannian manifold. If the eigenvalues of the curvature operator are nonnegative at every point, then $\chi(X) \ge 0$.
	\end{proposition}
	
	We are now ready to prove our main theorem.
	
	\begin{proof}[\textbf{Proof of Theorem~\ref{thm:main}}]
		Suppose, for the sake of contradiction, that the statement is false. Then there exist $\Lambda > 0$ and a sequence $\{\varepsilon_i\}_{i=1}^\infty$ of positive numbers such that:
		\begin{enumerate}
			\item[(1)] $\varepsilon_i \to 0$ as $i \to \infty$;
			\item[(2)] for each $i$, there exists a connected closed $2n$-manifold $X_i$ equipped with a Riemannian metric $g_i$ such that
			\[
			-\varepsilon_i \le \lambda_1(X_i, g_i) \cdot \operatorname{diam}^2(X_i, g_i) \le \cdots \le \lambda_{n(2n-1)}(X_i, g_i) \cdot \operatorname{diam}^2(X_i, g_i) \le \Lambda,
			\]
			and $\chi(X_i) < 0$.
		\end{enumerate}
		
		After rescaling each metric, we may assume that
		\[
		\operatorname{diam}(X_i, g_i) = 1 \quad \text{for all } i.
		\]
		The two-sided bounds on the eigenvalues of the curvature operator then imply that
		\[
		|\lambda_k(X_i,g_i)| \le \Lambda \quad \text{for all } k,i,
		\]
		and Lemma~\ref{L1} yields a uniform lower bound $\operatorname{Vol}(X_i,g_i) \ge v(n,\Lambda) > 0$. Thus the sequence $(X_i,g_i)$ satisfies the hypotheses of Theorem~\ref{thm:Cheeger-Gromov}. By passing to a subsequence, we obtain:
		\begin{enumerate}
			\item[(a)] a closed differentiable manifold $X$ endowed with a Riemannian metric $g_\infty$ of class $C^{1,\alpha}$ for all $0 < \alpha < 1$;
			\item[(b)] diffeomorphisms $f_i : X \to X_i$ such that $f_i^* g_i \to g_\infty$ in the $C^{1,\alpha}$-topology for all $\alpha \in (0,1)$.
		\end{enumerate}
		
		Replacing $(X_i, g_i)$ by $(X, f_i^* g_i)$, we may and shall henceforth assume that all metrics $g_i$ are defined on the same manifold $X$, that
		\[
		g_i \to g_\infty \quad \text{in } C^{1,\alpha} \text{ for all } \alpha \in (0,1),
		\]
		and that $\chi(X) = \chi(X_i) < 0$ for all $i$. Applying Lemma~\ref{lem:curvature-convergence}, and passing to a further subsequence if necessary, we may also assume that
		\[
		g_i \to g_\infty \quad \text{in } L^p_2(X) \text{ for all } p \in [1,\infty),
		\]
		and that the corresponding curvature tensors $K_i$ and $K_\infty$ satisfy
		\[
		K_i \to K_\infty \quad \text{in } L^p(X) \text{ for all } p \ge 1.
		\]
		In particular, $K_\infty \in L^p(X)$ for all $p > 1$.
		
		Let $\mathfrak{R}_i$ and $\mathfrak{R}_\infty$ denote the curvature operators of $(X,g_i)$ and $(X,g_\infty)$, viewed as self-adjoint endomorphisms of $\Lambda^2 TX$. The $L^p$-convergence of the curvature tensors implies that
		\[
		\|\mathfrak{R}_i - \mathfrak{R}_\infty\| \to 0 \quad\text{in }L^p(X)
		\]
		for all $p\ge1$, where $\|\cdot\|$ denotes the operator norm on each fiber of $\Lambda^2 TX$. Let $\lambda_1(g_i)\le\cdots\le\lambda_{N}(g_i)$ (with $N=n(2n-1)$) be the eigenvalues of $\mathfrak{R}_i$, and similarly for $\mathfrak{R}_\infty$. Applying Theorem~\ref{T3} pointwise and integrating, we obtain
		\[
		\max_{1\le k\le N}|\lambda_k(g_i) - \lambda_k(g_\infty)| \le \|\mathfrak{R}_i - \mathfrak{R}_\infty\|
		\]
		and hence
		\[
		\lambda_k(g_i) \to \lambda_k(g_\infty) \quad\text{in }L^p(X)
		\]
		for all $p\ge1$ and all $1\le k\le N$.
		
		Using the assumptions on the eigenvalues of the curvature operators of $g_i$ and the fact that $\operatorname{diam}(X,g_i)=1$, we have
		\[
		-\varepsilon_i \le \lambda_1(g_i) \le \cdots \le \lambda_N(g_i) \le \Lambda
		\]
		for all $i$. Combining the $L^p$-convergence with these bounds, we may pass to a further subsequence (still denoted by $i$) such that, for almost every $x\in X$ and all $k$,
		\[
		\lambda_k(g_i)(x) \to \lambda_k(g_\infty)(x).
		\]
		Letting $i\to\infty$ and using $\varepsilon_i\to0$, we conclude that
		\[
		0 \le \lambda_1(g_\infty)(x) \le \cdots \le \lambda_N(g_\infty)(x) \le \Lambda
		\]
		for almost every $x\in X$. In other words, the curvature operator of $g_\infty$ is almost everywhere nonnegative.
		
		On the other hand, by \cite{BK78} the pointwise norm of the curvature operator is controlled by the sectional curvature:
		\[
		|\lambda_k(g)(x)| \le C(n)\,|K(g)(x)|
		\]
		for all $x \in X$ and all $k$, where $C(n) > 0$ depends only on $n$. Consequently, the Euler characteristic can be computed from the limit metric $g_\infty$ as
		\[
		\chi(X) = \lim_{i \to \infty} \int_{X_i} P(K(X_i, g_i))\, d\mathrm{Vol}_{g_i}
		= \int_X P(K(X, g_\infty))\, d\mathrm{Vol}_{g_\infty},
		\]
		where $P$ is the Euler polynomial as in Lemma~\ref{L1}. Since the eigenvalues of the curvature operator of $g_\infty$ are nonnegative almost everywhere, Proposition~\ref{P2} implies that
		\[
		\chi(X) \ge 0,
		\]
		which contradicts the assumption $\chi(X) < 0$. This contradiction completes the proof of Theorem~\ref{thm:main}.
	\end{proof}

	% \section*{Acknowledgements}
	
	% \bigskip
	
\end{document}